\documentclass[12pt]{amsart}
\usepackage{amsmath,amssymb,amsthm,amscd,verbatim}
\usepackage[dvips]{graphicx}
\theoremstyle{plain}
\newtheorem{thm}{Theorem}
\newtheorem{lem}[thm]{Lemma}

\theoremstyle{definition}

\theoremstyle{remark}
\newtheorem{rem}[thm]{Remark}

\begin{document}
\baselineskip=19pt
\title{Generators for the mapping class group of a nonorientable surface}
\author{Susumu Hirose}
\address{Department of Mathematics,  
Faculty of Science and Technology, 
Tokyo University of Science, 
Noda, Chiba, 278-8510, Japan} 
\email{hirose\b{ }susumu@ma.noda.tus.ac.jp}
\thanks{This research was supported by Grant-in-Aid for 
Scientific Research (C) (No. 16K05156), 
Japan Society for the Promotion of Science. }
\begin{abstract} 
We show that Szepietowski's system of generators  
for the mapping class group 
of a non-orientable surface is a minimal generating set 
by Dehn twists and $Y$-homemorphisms. 
\end{abstract}
\maketitle

Let $N_g$ be a non-orientable surface which is a connected sum of $g$ 
projective planes. 
Let $\mathcal{M}(N_g)$ be the group of isotopy classes of homeomorphisms 
over $N_g$, i.e., the {\em mapping class group\/} of $N_g$. 
In this paper, we assume that $g \geq 4$. 

We introduce some elements of $\mathcal{M}(N_g)$. 
A simple closed curve $\gamma_1$ (resp. $\gamma_2$) in $N_g$ 
is {\em two-sided\/} (resp. {\em one-sided\/}) 
if a regular neighborhood of $\gamma_1$ (resp. $\gamma_2$) 
is an annulus (resp. M\"{o}bius band).  
For a two-sided simple closed curve $\gamma$ on $N_g$,
we denote by $t_{\gamma}$ a Dehn twist about $\gamma$. 
We indicate the direction of a Dehn twist by an arrow beside the curve 
$\gamma$ as shown in Figure \ref{fig:Dehn-twist}. 
For a one-sided simple closed curve $m$ and a two-sided 
simple closed curve $a$ which intersect transversely in one point, 
let $K \subset N_g$ be a regular neighborhood of $m \cup a$, 
which is homeomorphic to the Klein bottle with one boundary component. 
Let $M$ be a regular neighborhood of $m$. 
We denote by $Y_{m,a}$ a homeomorphism over $N_g$ which is described 
as the result of pushing $M$ once along $a$ keeping 
the boundary of $K$ fixed 
(see Figure~\ref{fig:Y-homeo}). 
We call $Y_{m,a}$ a $Y$-homeomorphism. 
\begin{figure}[hbtp]
\includegraphics[height=2cm]{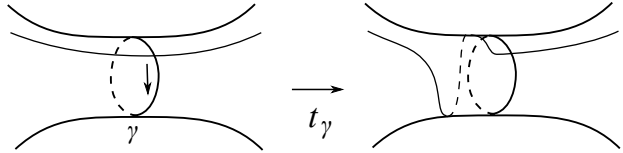}
\caption{The direction of $t_\gamma$ is indicated by an arrow beside $\gamma$. }
\label{fig:Dehn-twist}
\end{figure}
\begin{figure}[hbtp]
\includegraphics[height=4cm]{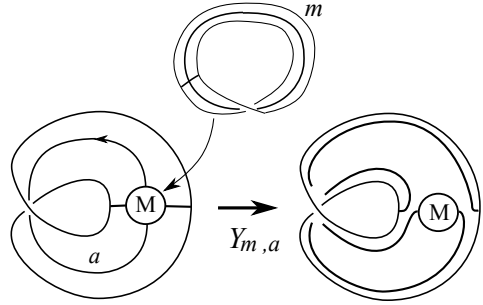}
\caption{ A cirlcle with ``M" indicates a place where to attach a M\"obius band.}
\label{fig:Y-homeo}
\end{figure}

Lickorish showed that $\mathcal{M}(N_g)$ is generated by 
Dehn twists and $Y$-homeomorphisms \cite{Lickorish1}, and 
that $\mathcal{M}(N_g)$ is not generated by Dehn twists \cite{Lickorish2}. 
Furthermore, Chillingworth \cite{Chillingworth} found a finite system of 
generators for $\mathcal{M}(N_g)$. 
Birman and Chillingworth \cite{BC} obtained a finite system of generators 
by using an argument on the orientable two fold covering of $N_g$. 
Szepietowski \cite{Szepietowski1} reduced the system of Chillingworth's generators 
for  $\mathcal{M}(N_g)$ and showed: 
\begin{figure}[hbtp]
\includegraphics[height=6.5cm]{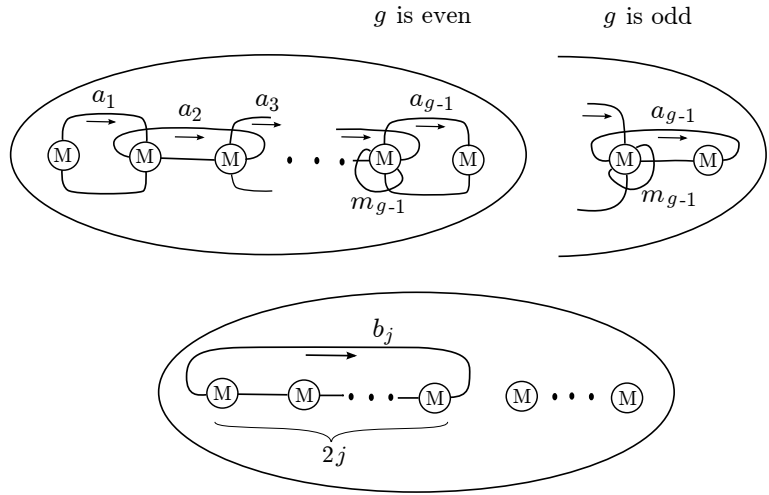}
\caption{Chillingworth's generators for $\mathcal{M}(N_g)$. }
\label{fig:generator}
\end{figure}

\begin{thm}\cite[Theorem 3.1]{Szepietowski1}
$\mathcal{M}(N_g)$ is generated by $t_{a_i}$ $(i=1, \ldots, g-1)$, 
$t_{b_2}$ and $Y_{m_{g-1}, a_{g-1}}$, 
where $a_i, m_{g-1}, b_2$ are simple closed curves shown in 
Figure \ref{fig:generator}.  
\end{thm}

On the other hand, Lickorish \cite{Lickorish} showed that 
the mapping class group $\mathcal{M}(\Sigma_g)$ of the orientable 
closed surface $\Sigma_g$ of genus $g$ is generated by finitely many Dehn twists, 
and Humphries \cite{Humphries} reduced the number of Dehn twists 
generating $\mathcal{M}(\Sigma_g)$ to $2g+1$ and 
showed that this is the minimum 
number of Dehn twists generating $\mathcal{M}(\Sigma_g)$. 
We will show the analogous result for the mapping class group of the non-orientable surface. 

\begin{thm} \label{thm:minimal}
We assume  $g\geq 4$. 
If Dehn twists $t_{c_1}, \ldots, t_{c_n}$ and 
$Y$-homeomorphisms $Y_1, \ldots, Y_k$ 
generate $\mathcal{M}(N_g)$, then $n \geq g$ and $k \geq 1$. 
In particular,  any proper subset of $\{  t_{a_i}  (i=1, \ldots, g-1),  t_{b_2}, Y_{m_{g-1}, a_{g-1}} \}$ 
does not generate $\mathcal{M}(N_g)$. 
\end{thm}

\begin{rem}
When $g=1$, $\mathcal{M}(N_g)$ is trivial. 
When $g=2$, $\mathcal{M}(N_2) \cong \mathbb{Z}_2 \times \mathbb{Z}_2$ 
and generated by $t_{a_1}$ and $Y_{m_1, a_1}$ (see \cite[Lemma 5]{Lickorish1}), 
therefore $\{ t_{a_1}, Y_{m_1, a_1} \}$ is a minimal generating set by Dehn twists 
and $Y$-homeomorphisms. 
When $g=3$, $\mathcal{M}(N_3)$ 
is generated by $t_{a_1}, t_{a_2}$ and $Y_{m_2, a_2}$
(see \cite[Theorem 3]{BC} and \cite[Theorem 3.1]{Szepietowski1}). 
If $\mathcal{M}(N_3)$ is generated by one Dehn twist $t_a$ and one $Y$-homeomorphism, 
then the group of the action of $\mathcal{M}(N_3)$ on $H_1(N_3;\mathbb{Z}_2)$ 
should be isomorphic to $\mathbb{Z}_2$ generated by the induced isomorphism 
$(t_a)_*$ on $H_1(N_3;\mathbb{Z}_2)$. 
Nevertheless, $(t_{a_1})_*$ is not equal to  $(t_{a_2})_*$. 
Therefore $\{ t_{a_1}, t_{a_2}, Y_{m_2, a_2}\}$ is 
a minimal generating set by Dehn twists 
and $Y$-homeomorphisms.  \end{rem}

Let $w_1 : H_1(N_g;\mathbb{Z}_2) \to \mathbb{Z}_2$ be 
the first Stiefel-Whitney class, that is to say, 
if $x \in H_1(N_g;\mathbb{Z}_2)$ is represented by a one-sided simple 
closed curve on $N_g$ then $w_1(x) = 1$, 
otherwise $w_1(x) = 0$. 
For the basis $\{ x_1, \ldots, x_g \}$ for $H_1(N_g ; \mathbb{Z}_2)$ 
indicated in Figure \ref{fig:basis}, $w_1(x_i) = 1$. 
For each pair of elements $x$, $y$ of $H_1(N_g; \mathbb{Z}_2)$, 
the $\mathbb{Z}_2$-intersection form of $x$ and $y$ is denoted by  $(x,y)$. 
For the basis $\{ x_1, \ldots, x_g \}$, $(x_i, x_j) = \delta_{i,j}$.  
Let $H_1^+(N_g;\mathbb{Z}_2)$ be the kernel of $w_1$, 
then $\dim_{\mathbb{Z}_2}   H_1^+(N_g;\mathbb{Z}_2) = g-1$. 
If a complement of a two-sided simple closed curve $c$ on $N_g$ is connected 
and non-orientable, we call $c$ an {\em admissible} $A$-circle. 
For an admissible $A$-circle $c$ on $N_g$, 
$N_g \setminus c$ is homeomorphic to $N_{g-2}$ removed two 2-disks. 
Therefore, if $c_1$ and $c_2$ are admissible $A$-circles then 
there is $\phi \in \mathcal{M}(N_g)$ such that $\phi(c_1) = c_2$, 
by the change of coordinates principle in \cite[\S 1.3]{Farb-Margalit}.  

\begin{figure}[hbtp]
\includegraphics[height=3.5cm]{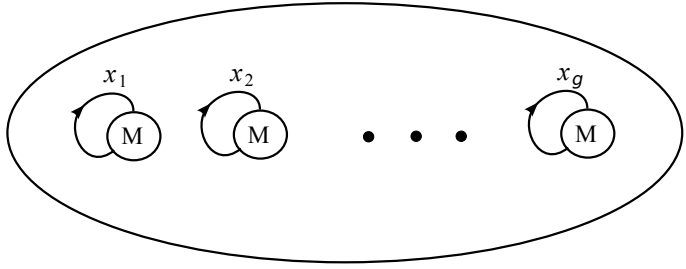}
\caption{A basis for $H_1(N_g ; \mathbb{Z}_2)$. }
\label{fig:basis}
\end{figure}

\begin{lem} \label{lem:non-admissible}
Let $c$ be a two-sided simple closed curves on $N_g$. 
If $c$ is not admissible, then $c$ represent $0$ or $x_1 + \cdots + x_g$ 
in $H_1(N_g;\mathbb{Z}_2)$. 
\end{lem}
\begin{proof} 
If $c$ is not admissible, then either $N_g \setminus c$ is not connected 
or $N_g \setminus c$ is connected and orientable. 
In the former case, $c$ is $0$ in $H_1(N_g;\mathbb{Z}_2)$.  
In the latter case, $g$ is even, and 
there is a homeomorphism which brings $c$ to $b_{g/2}$ in Figure \ref{fig:generator}, 
since their complements are homeomorphic to $\Sigma_{g/2-1}$ removed two $2$-disks. 
The simple closed curve $b_{g/2}$ represents $x_1 + \cdots + x_g$ and 
the action of any homeomorphism of $N_g$ on $H_1(N_g;\mathbb{Z}_2)$ 
preserves $x_1 + \cdots + x_g$. 
Therefore $c$ represents $x_1 + \cdots + x_g$.  
\end{proof}

\begin{lem} \label{lem:one-admissible}
If $t_{c_1}, \ldots, t_{c_n}$ and $Y$-homeomorphisms 
$Y_1, \ldots, Y_k$ generate 
$\mathcal{M}(N_g)$, then at least one of $c_1, \ldots, c_n$ 
is admissible. 
\end{lem}
\begin{proof}
For $y \in H_1(N_g;\mathbb{Z}_2)$, 
we define an isomorphism $\tau_y$ of  $H_1(N_g;\mathbb{Z}_2)$ by 
$\tau_y (x) = x + (x,y) y$. 
By Lemma \ref{lem:non-admissible} and the fact that $Y$-homeomorphisms acts 
on $H_1(N_g;\mathbb{Z}_2)$ trivially, 
if $c_1, \ldots, c_n$ are not admissible, then 
the action of each elements of $\mathcal{M}(N_g)$ on $H_1(N_g;\mathbb{Z}_2)$ is  
a power of $\tau_{x_1 + \cdots + x_g}$. 
On the other hand, $(t_{a_1})_* = \tau_{x_1 + x_2}$ is not a power of  
$\tau_{x_1 + \cdots + x_g}$. 
\end{proof}

\begin{lem} \label{lem:lowerbound}
If $t_{c_1}, \ldots, t_{c_n}$ and $Y$-homeomorphisms 
$Y_1, \ldots, Y_k$ generate 
$\mathcal{M}(N_g)$ then $[c_1], \ldots, [c_n]$ generate $H_1^+(N_g;\mathbb{Z}_2)$.  
In particular, $n \geq g-1$. 
\end{lem}

\begin{figure}[hbtp]
\includegraphics[height=2cm]{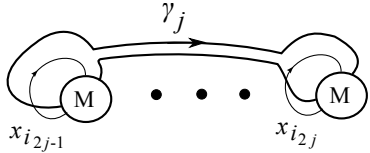}
\caption{An element $x_{i_{2j-1}}+ x_{i_{2j}}$ $ \in H_1(N_g ; \mathbb{Z}_2)$ 
is represented by an admissible $A$-circle $\gamma_i$}
\label{fig:repre-gamma}
\end{figure}

\begin{proof}
By Lemma \ref{lem:one-admissible}, we may assume $c_1$ is an admissible $A$-circle. 
For any $x \in H_1^+(N_g;\mathbb{Z}_2)$, 
we can write $x = x_{i_1} + x_{i_2} + \cdots + x_{i_{2k}}$. 
We can represent $x_{i_{2j-1}}+ x_{i_{2j}}$ by 
an admissible $A$-circle $\gamma_j$ as in Figure \ref{fig:repre-gamma}.  
Hence, $x$ is represented by a union of admissible $A$-circles, 
that is, $x = [\gamma_1]+ \cdots + [\gamma_k]$ in $H_1^+(N_g;\mathbb{Z}_2)$. 
For each $\gamma_j$, there is an element $\phi_j \in \mathcal{M}(N_g)$ 
such that $\phi_j(c_1) = \gamma_j$. 
By the assumption of this lemma, 
$\phi_j $ is a product of $t_{c_1}, \ldots, t_{c_n}$ and $Y_1, \ldots, Y_k$. 
We see that $Y_i$ acts on $H_1(N_g ; \mathbb{Z}_2)$ trivially, and, 
for each $x \in  H_1(N_g ; \mathbb{Z}_2)$, 
$(t_{c_i})_* (x) = x + (x,[c_i]) [c_i] $. 
Therefore, $[\gamma_j] \in H_1^+(N_g;\mathbb{Z}_2)$ is a sum 
of  $[c_1], \ldots, [c_n]$, hence $x$ is a sum of $[c_1], \ldots, [c_n]$. 
This shows that $H_1^+(N_g;\mathbb{Z}_2)$ is generated by 
$[c_1], \ldots, [c_n]$. 
\end{proof}

Let $2\times : \mathbb{Z}_2 \to \mathbb{Z}_4$ be an injection 
defined by $2\times([n])=[2n]$. 
A map $q : H_1(N_g;\mathbb{Z}_2) \to \mathbb{Z}_4$ is 
called a $\mathbb{Z}_4$-{\em quadratic form\/}, 
if $q(x+y) = q(x)+q(y) + 2 \times (x,y)$ for any $x, y \in H_1(N_g;\mathbb{Z}_2) $.  
This map $q$ is determined by values of $q$ for elements in a $\mathbb{Z}_2$-basis of 
$H_1(N_g;\mathbb{Z}_2)$. 
Putting $x=y=0$ in the above formula, we have $q(0)=0$. 
If $x \in H_1(N_g ; \mathbb{Z}_2)$ is represented by a one-sided simple closed curve, 
in other word $x$ is represented by a core of a M\"{o}bius band embedded in $N_g$, 
then $(x,x)=1$. 
Since $2x= 0$ in $H_1(N_g ; \mathbb{Z}_2)$, we have 
$0 = q(x + x) = q(x) + q(x) + 2 \times (x,x) = 2 q(x) + 2$. 
Therefore, we have $q(x) = \pm 1$. 
By the same argument, 
if $x \in H_1(N_g ; \mathbb{Z}_2)$ is represented by a two-sided simple closed curve, 
then we have $q(x)= 0 $ or $2$. 
\begin{lem} \label{lem:no-quadratic-form}
There is no $\mathbb{Z}_4$-quadratic form over $H_1(N_g;\mathbb{Z}_2)$ 
which is preserved by every non-trivial element of $\mathcal{M}(N_g)$. 
\end{lem}
\begin{proof}
For any $\mathbb{Z}_4$-quadratic form $q$ over $H_1(N_g;\mathbb{Z}_2)$, 
there is a non-trivial  element $x \in H_1^+(N_g;\mathbb{Z}_2)$ such that 
$q(x)=0$; even if $q([a_1])=q([a_3])=2$ then 
$q([a_1]+[a_3]) = q([a_1])+q([a_3]) + 2\times([a_1],[a_3]) = 0$. 
When $x = x_{i_1} + \ldots + x_{i_{2n}}$, let $y = x_{i_1}$, 
then $(y,x) = 1$. 
Let $\gamma$ be a simple closed curve on $N_g$ representing $x$ 
then $q \circ (t_{\gamma})_* (y) = q( y + (y,x) x) 
= q(y) + q((y,x)x) + 2\times(y, (y,x)x) 
= q(y) + q(x) + 2 = q(y) + 2 \not= q(y)$. 
Therefore $q \circ (t_{\gamma})_* \not= q$. 
\end{proof}
\begin{lem}\label{invariant-quadratic-form}
Let $c_1, \ldots, c_{g-1}$ be two-sided simple closed curves such that 
$[c_1], \ldots, [c_{g-1}]$ generate $H_1^+(N_g;\mathbb{Z}_2)$, 
then there is a $\mathbb{Z}_4$-quadratic form over $H_1(N_g;\mathbb{Z}_2)$ 
preserved by any $t_{c_i}$. 
\end{lem}
\begin{proof}
Let $\alpha$ be a one-sided simple closed curve on $N_g$, 
then $\{ [c_1], \ldots, [c_{g-1}], [\alpha] \}$ is a $\mathbb{Z}_2$-basis of 
$H_1(N_g;\mathbb{Z}_2)$. 
We define a $\mathbb{Z}_4$-quadratic form $q$ over $H_1(N_g;\mathbb{Z}_2)$ 
by $q([c_1])= \cdots = q([c_{g-1}])=2$ and $q([\alpha])=1$. 
For any $i = 1, \ldots, g-1$ and $x \in H_1(N_g;\mathbb{Z}_2)$, 
we see $q \circ {(t_{c_i})_*} (x) = q(x + (x,[c_i])[c_i]) 
= q(x) + q((x,[c_i])[c_i]) + 2 \times (x,(x,[c_i])[c_i]) $. 
If $(x,[c_i])=0$, $q \circ {(t_{c_i})_*} (x) = q(x)$. 
If $(x,[c_i])=1$, $q \circ {(t_{c_i})_*} (x) = 
q(x) + q([c_i]) + 2 \times (x,[c_i]) = q(x) + 2 +2 = q(x)$. 
Therefore, $q \circ {(t_{c_i})_*} = q$ for any $i = 1, \ldots, g-1$. 
\end{proof}
We assume that Dehn twists $t_{c_1}, \ldots, t_{c_n}$  
and $Y$-homeomorphisms $Y_1, \ldots, Y_k$ generate $\mathcal{M}(N_g)$. 
In \cite{Lickorish2}, Lickorish showed that $\mathcal{M}(N_g)$ is not generated by 
Dehn twists, therefore we see $k \geq 1$. 
By Lemma \ref{lem:lowerbound}, $[c_1], \ldots, [c_n]$ generate 
$H_1^+(N_g ; \mathbb{Z}_2)$, in particular $n \geq g-1$.   
We assume that $n = g-1$. 
By Lemma \ref{invariant-quadratic-form}, there is a $\mathbb{Z}_4$-quadratic form 
over $H_1(N_g;\mathbb{Z}_2)$ preserved by Dehn twists $t_{c_1}, \ldots, t_{c_n}$ 
and $Y$-homeomorphisms $Y_1, \ldots, Y_k$, 
which contradicts Lemma \ref{lem:no-quadratic-form}. 
Hence, we see $n \geq g$. 
This completes the proof of Theorem \ref{thm:minimal}.  

\begin{rem}
The proof of Theorem \ref{thm:minimal} is inspired by the master thesis 
\cite{Ishimura} by Shigehisa Ishimura, in which he proved Humphries' result by using 
$\mathbb{Z}_2$-quadratic form over $H_1(\Sigma_g;\mathbb{Z}_2)$. 
\end{rem}
\subsection*{Acknowledgments}
The author would like to thank Mustafa Korkmaz and Genki Omori for their useful comments. 


\begin{thebibliography}{99}
%
%
\bibitem{BC}{\bf J.S.~Birman and D. R. J. ~Chillingworth}, 
{\it On the homeotopy group of a non-orientable surface\/}, 
Math.\ Proc.\ Camb.\ Phil.\ Soc.\ 71 (1972),\ 437--448. 
Erratum: Math.\ Proc.\ Camb.\ Phil.\ Soc.\ 136 (2004),\ 441--441.
%
\bibitem{Chillingworth}{\bf D. R. J. ~Chillingworth},
{\it A finite set of generators for the homeotopy group of
a non-orientable surface\/},
Proc.\ Camb.\ Phil.\ Soc.\ 65 (1969),\ 409--430
%
\bibitem{Farb-Margalit}{\bf B.~Farb and D.~Margalit}, 
A primer on mapping class groups. 
Princeton Mathematical Series, 49. Princeton University Press, Princeton, NJ, 2012. 
%
\bibitem{Humphries}{\bf S. P. ~Humphries}, 
{\it Generators for the mapping class group\/},  
In: Topology of low-dimensional manifolds (Proc. Second Sussex Conf., Chelwood Gate, 1977), 
44--47, Lecture Notes in Math., 722, Springer, Berlin, 1979.
%
\bibitem{Ishimura}{\bf S. ~Ishimura}, 
{\it The spin structures over surfaces and the number of generators for 
the mapping class group\/}, (Japanese) Master Thesis, Osaka City University,  2004. 
%
\bibitem{Lickorish} {\bf W.B.R. ~Lickorish}, 
{\it A finite set of generators for the homeotopy group of a 2-manifold\/},  
Proc.\ Cambridge\ Philos.\ Soc. 60(1964),\ 769--778, 
Erratum: Proc.\ Cambridge\ Philos.\ Soc. 62(1966),\ 679--681. 
%
\bibitem{Lickorish1}{\bf W.B.R. ~Lickorish}, 
{\it Homeomorphisms of non-orientable two-manifolds\/}, 
Proc.\ Cambridge\ Philos.\ Soc. 59(1963),\ 307--317.
%
\bibitem{Lickorish2}{\bf W.B.R. ~Lickorish}, 
{\it On the homeomorphisms of a non-orientable surface\/}, 
Proc.\ Cambridge\ Philos.\ Soc. 61(1965),\ 61--64. 
%
\bibitem{Szepietowski1}{\bf B. ~Szepietowski}, 
{\it A finite generating set for the level 2 
mapping class group of a nonorientable surface\/}, 
Kodai\ Math.\ J.\ 36(2013),\ 1--14
%
\end{thebibliography}
\end{document}